\documentclass[12pt]{amsart}

\usepackage[utf8]{inputenc}

\usepackage{amsfonts,amsmath,amscd,amssymb}

\usepackage{epsfig,graphics}

\usepackage{hyperref}

\usepackage{xcolor}
\usepackage{soul}

\newtheorem{theorem}{Theorem}[section]
\newtheorem{proposition}[theorem]{Proposition}
\newtheorem{lemma}[theorem]{Lemma}
\newtheorem{corollary}[theorem]{Corollary}
\newtheorem{conjecture}[theorem]{Conjecture}

\newtheorem{definition}[theorem]{Definition}

\newtheorem{remark}[theorem]{Remark}

\newtheorem*{prop*}{Proposition}

\newtheorem*{thm*}{Theorem}

\newtheorem*{cor*}{Corollary}

\theoremstyle{definition}

\newtheorem*{defn*}{Definition}

\newtheorem*{rem*}{Remark}

\newtheorem{ob}[theorem]{Observation}

\providecommand{\tightlist}{%
  \setlength{\itemsep}{0pt}\setlength{\parskip}{0pt}}

\newcommand{\BR}{{\mathbb{R}}}     
\title{Uniformly affine actions on Banach spaces: growth of cocycles }
\author[K. Boucher]{Kevin Boucher}
\address{}
\email{kevin.boucher01@gmail.com}
\author[G. Grützner]{Georg Grützner}
\address{Department of Mathematics, University of Luxembourg, L-4364 Esch-sur-Alzette}
\email{georg.gruetzner@uni.lu}
\begin{document}
\maketitle
\begin{abstract}
We investigate growth properties of cocycles with values in uniformly bounded representations on super-reflexive Banach spaces; this includes $L^p$-spaces for $1<p<\infty$ as well as Hilbert spaces.
We then study the generalized Hilbert compression of cocycles arising in this setting for the Property (T) groups $\mathrm{Sp}(n,1)$, $n\ge 2$, and establish the existence of uniformly Lipschitz affine actions with optimal growth.\end{abstract}
\subsubsection*{Notation}

We write \(\lesssim\) to mean inequality up to a positive multiplicative
constant, and the corresponding equivalence is denoted
\(\asymp\).
Where it is more suitable we fall back to Landau
  notation i.e.~given a function
  \(g: \mathbb{R}_{>0} \to \mathbb{R}_{>0}\), we write \(o(g(s))\) as
  \(s \to 0^+\) or \(+\infty\) for any term depending on \(s > 0\) such
  that \(\lim_{s \to 0^+\text{ or } +\infty} \frac{o(g(s))}{g(s)} = 0\).
For a topological space \(X\), and
\(f, h: X \rightarrow \mathbb{R}_{+}\), we write \(f \preceq h\) if
there exists a constant \(M > 0\) and a compact subset \(K \subset X\)
such that \(f \leq M h\) outside \(K\). We write \(f \sim h\) if
\(f \preceq h \preceq f\). We write \(f \prec h\) if, for every
\(\varepsilon>0\), there exists a compact subset \(K \subset G\) such
that \(f \leq \varepsilon h\) outside \(K\).

\section{Introduction}
\label{sec:introduction}

The study of isometric group actions on Hilbert spaces has deep roots in
geometric group theory and representation theory.

A central notion in this context is \emph{Kazhdan's Property (T)},
introduced by Kazhdan in 1967 as a rigidity property for locally compact
groups. A group with Property (T) is characterized by the fact that
every unitary representation admitting almost invariant vectors must
contain a non-zero invariant vector \cite{Kazhdan1967}.

An equivalent cohomological formulation is provided by the
\emph{Delorme--Guichardet theorem}: a \(\sigma\)-compact locally compact
group has Property (T) if and only if its first cohomology group with
coefficients in any unitary representation is trivial. In geometric
terms, this means that every affine isometric action of the group on a
Hilbert space has a global fixed point \cite{Delorme1977}
\cite{Guichardet1977}.

In sharp contrast with Kazhdan groups, such as higher-rank semisimple
Lie groups and their lattices, Haagerup showed in his landmark 1978  work \cite{Haagerup1978}
that free groups admit proper affine isometric actions on Hilbert
spaces. This result initiated the study of what is now known as the
\emph{Haagerup property} (or \emph{a-T-menability} in Gromov's
terminology) \cite{Haagerup1978}. The Haagerup property provides a
natural obstruction to Property (T). One of its major
consequences is its connection with the Baum--Connes conjecture: Higson
and Kasparov showed in the early 2000s that groups with the Haagerup
property satisfy the Baum--Connes conjecture \cite{Higson2001}.

In a different direction, Cornulier, Tessera, and Valette introduced in
\cite{Cornulier2007} the notion of the \emph{equivariant Hilbert
compression exponent} of a locally compact compactly generated group
\(G\), endowed with a word metric: \[\begin{aligned}
\alpha_{2}^*(G)=\sup \left\{\right. \alpha \geq 0: &\exists \text { unitary representation } \pi,\\  
&\exists  b \in Z^1(G, \pi) \text{ s.t. } \left.\|b(g)\| \succeq|g|_S^\alpha\right\}\end{aligned}\]

This invariant quantifies the asymptotic growth of orbits of affine
isometric actions on Hilbert spaces via the growth of the associated
1-cocycles (see Subsection \ref{sec:statement_of_the_results} for
precise definitions). The equivariant compression exponent refines the
earlier, non-equivariant notion of Hilbert compression introduced by
Guentner and Kaminker \cite{Guentner2004} in the context of uniform
embeddings of metric spaces.

The value of \(\alpha_2^*(G)\) reflects the quality of equivariant
embeddings into Hilbert space. An exponent \(\alpha_2^*(G)=1\) indicates
the existence of embeddings that are arbitrary close to being
quasi-isometric, whereas \(\alpha_2^*(G)<1\) means that some contraction
is unavoidable. A striking dichotomy emerges: for \emph{amenable
groups}, equivariant and non-equivariant Hilbert compressions coincide (a
phenomenon often referred to as \emph{Gromov's trick}), and in many
cases satisfy \(\alpha_2^*(G) > \tfrac12\) \cite{Gournay2016}. By
contrast, for \emph{non-amenable groups}, the equivariant Hilbert
compression exponent is bounded above by \(\tfrac12\).

More precisely, if a group \(G\) admits a Hilbert-valued 1-cocycle with
compression exponent strictly greater than \(\tfrac12\), then \(G\) must
be amenable. Equivalently, non-amenable groups cannot admit cocycles
whose growth exceeds \(\|b(g)\| \sim |g|_S^{1/2}\). Thus, \(\tfrac12\)
appears as a critical threshold for equivariant Hilbert compression. For
instance, the optimal Hilbert cocycle on the free group \(F_2\) exhibits
square-root growth, \[\|b(g)\| \sim |g|_{S}^{1/2},\] achieving the
maximal possible exponent \(\tfrac12\) in the non-amenable setting. For
non-amenable groups, equivariance plays a crucial role. Indeed, Tessera
showed that all word-hyperbolic groups admit non-equivariant Hilbert
embeddings with compression exponent equal to \(1\) \cite{Tessera2011},
demonstrating a sharp contrast between equivariant and non-equivariant
notions of compression.

Moving beyond isometric actions, we investigate cohomology with
coefficients in \emph{uniformly bounded representations}. This broader
framework is motivated by a result of Pisier, who showed that there
exist uniformly bounded representations of non-amenable groups that are
not equivalent to any unitary representation \cite{Pisier2001}.

This perspective is further reinforced by the following conjecture of Shalom:

\begin{conjecture}[Shalom{}]
Once the requirement that the linear part be unitary is dropped, every
Gromov hyperbolic group, including those having Property (T), admits a
proper uniformly Lipschitz affine action on a Hilbert space.
\end{conjecture}

These ideas suggest that the class of uniformly Lipschitz affine actions
on Hilbert space is richer than that of isometric actions,
and that the associated compression may differ.

Recent progress toward this conjecture has been remarkable. Nishikawa
\cite{Nishikawa2020_preprint} verified Shalom's conjecture for the
rank-one simple Lie groups \(\mathrm{Sp}(n,1)\), \(n \geq 2\),
constructing for each such group a metrically proper affine action on a
Hilbert space whose linear part is uniformly bounded but not unitary.
These were the first examples of Property (T) groups admitting proper
affine actions on Hilbert spaces.

In a different direction, moving beyond Hilbert spaces, Druţu, Mackay,
and Vergara showed that many groups known to lack proper isometric
actions on non-reflexive Banach spaces such as \(\ell^1\) admit proper
actions by uniformly Lipschitz affine transformations
\cite{Drutu2023_preprint} \cite{Vergara2023}. These results further
support the idea that allowing limited flexibility, uniformly bounded or
Lipschitz actions instead of strict isometries, can dramatically improve
the large-scale geometry of group orbits.

\subsection{Statement of the results}
\label{sec:statement_of_the_results}

Let \(G\) be a locally compact topological group, \(B\) a Banach space and
\(\operatorname{GL}(B)\) the group of invertible continuous linear
transformations of \(B\). A \emph{(Banach) representation} of \(G\) is a
group homomorphism \[\pi: G \to \operatorname{GL}(B)\] that is
continuous w.r.t. the strong operator topology on
\(\operatorname{GL}(B)\).

A continuous map \(b: G \rightarrow B\) such that
\[b(g h)=b(g)+\pi(g) b(h), \quad \text { for all } g, h \in G\] is
called a \emph{cocycle with coefficients in \((\pi,B)\)}.
The space \(Z^1(\pi,B)\) of all cocycles with coefficients in
\((\pi,B)\), equipped with the compact-open topology, is a real
topological vector space under the pointwise operations.
A cocycle \(b: G \rightarrow B\) for which there exists \(v \in B\) such
that \[b(g)=\pi(g) v- v, \text { for all } g \in G,\] is called a
\emph{coboundary with respect to \(\pi\)}.
The set \(B^1(\pi,B)\) of all coboundaries w.r.t. \(\pi\) is a subspace
of \(Z^1(\pi,B)\).
The quotient vector space \[H^1(\pi,B)=Z^1(\pi,B) / B^1( \pi,B)\] is
called the \emph{first cohomology group with coefficients in
\((\pi,B)\)}.

A Banach representation \((\pi,B)\) of \(G\) is \emph{uniformly
bounded} if \[\sup_{g \in G} \|\pi(g)\|_{op} < \infty.\]

A Banach representation \((\pi ,B)\) is \emph{isometric} if it restricts
to a continuous group homomorphism \[\pi: G \to \operatorname{O}(B),\]
where \(\operatorname{O}(B) < \operatorname{GL}(B)\) is the subgroup of
linear isometric operators of \(B\).

\subsubsection{Results}
\label{sec:results}

Let \(G\) be a locally compact compactly generated group and let
\(S \subset G\) be an \emph{open relatively compact symmetric generating
set}. Then for each \(g \in G\), the \emph{word length} (or \emph{word
norm}) of \(g\) with respect to \(S\) is defined as
\[|g|_S = \min \{ n \in \mathbb{N} : g \in S^{n} \}.\]

\begin{ob}
\label{466efb}
For a given compactly generated, locally compact group and any 
 cocycle, $b$, with coefficients in
some uniformly bounded Banach representation, \((\pi,B)\), the \emph{Banach length function} defined by \(L: g\in G \mapsto\|b(g)\|\in\BR_+\)
 is Lipschitz w.r.t. any word norm on \(G\),
i.e.~for any open, relatively compact generating subset \(S \subset G\),
there exists \(C > 0\), s.t. \(\forall g \in G, L(g) \leq C|g|_{S}\).
\end{ob}

The modulus of uniform smoothness of a Banach space \(B\) is defined for
\(\tau>0\) as
\[\rho_B(\tau)=\sup \left\{\frac{\|u+v\|+\|u-v\|}{2}-1: \|u\| \leq 1 \text{ and } \|v\| \leq \tau \right\}.\]

If \(\lim _{\tau \rightarrow 0} \frac{\rho_{B}(\tau)}{\tau}=0\), then
\(B\) is said to be \emph{uniformly smooth}. If in addition there exists
a constant \(K\) such that \(\rho_B(\tau) \leq K \tau^p\) for some
\(p>1\), \(p \leq 2\), then \(B\) is said to be \emph{\(p\)-uniformly
smooth}.

\begin{remark}
\label{d20809}
A deep theorem of Pisier \cite{Pisier1975} states that if \(B\) is
uniformly smooth, then there exists some \(1<p \leq 2\) such that \(B\)
admits an equivalent norm that is \(p\)-uniformly smooth. In fact, up to
isomorphism, this class coincides with super-reflexive Banach spaces
\cite[Theorem~A.6]{Benyamini2000}.
\end{remark}

Generalizing ~\cite[Theorem~1.2]{Cornulier2007} beyond cocycles with coefficients in unitary representations, we establish the following amenability criterion:

\begin{theorem}
\label{da6dad}
Let \(\Gamma\) be a finitely generated, discrete group and \(S\) a
finite generating set of \(\Gamma\). Let \((\pi,B)\) be uniformly bounded representation on a \(p\)-uniformly
smooth Banach space \(B\). If \(\Gamma\) admits a cocycle, \(b\), with
coefficients in  \((\pi,B)\)
such that \[\|b(g)\| \succ |g|_{S}^\frac{1}{p},\] then \(\Gamma\) is amenable.
\end{theorem}

As far as the authors know this result is new even for
isometric Banach representations  on $L^r$-spaces, $r\neq1,\infty$, and improves the result of Naor and
Peres \cite{Naor2008} in the critical case when the equivariant
compression exponent is equal to \(\frac{1}{p}\).
As a consequence we obtain that Nica's construction in \cite{Nica2013} of affine actions for non-elementary hyperbolic groups on $L^p$-spaces with $p\neq 1,\infty$ large, has optimal growth.

Theorem \ref{da6dad} also implies, as in the unitary case, that the best possible behavior one
can expect for a cocycle \(b\) with coefficients in a uniformly bounded
Hilbert representation is \[\|b(g)\| \sim |g|_{S}^{1/2}.\]

Extending the list of groups admitting affine actions on Hilbert spaces with optimal growth discussed in \cite[Proposition~3.9]{Cornulier2007}, we prove
the following theorem:

\begin{theorem}
\label{6ea617}
Let \(G\) be a connected, simple, real rank \(1\) Lie group with finite
center and let \((\pi,H)\) be a uniformly bounded Hilbert representation
with a spectral gap. If \(H^1(\pi,H)\neq0\), then there exists
\(b \in Z^{1}(\pi,H)\) such that \[\|b\|\sim d^\frac{1}{2}(\cdot,e),\]
where \(d\) stands for the orbital distance induced by the canonical geometry of \(G/K\) with
\(K\) a compact maximal of \(G\).
\end{theorem}

\begin{remark}
A uniformly bounded Hilbert representation has a spectral gap if its
restriction to a complement of its invariant vectors admits no almost
invariant vectors. For precise definitions, see Section
\ref{sec:a_smooth_version_of_the_harmonic_representation_of_cocycles}.
We remark that all definitions given there work equally well for
uniformly bounded representations.
\end{remark}

Using Nishikawa's construction \cite{Nishikawa2020_preprint} we conclude
the paper with the following Corollary, proved in Section
\ref{sec:proof_of_corollary_ref99b6bf}.

\begin{corollary}
\label{99b6bf}
For all Kazhdan groups \(\mathrm{Sp}(n,1)\) with \(n\ge2\), there exists
a uniformly bounded Hilbert representation \((\pi,H)\) and a cocycle
\(b \in Z^{1}(\pi,H)\), s.t. \[\|b\|\sim d^\frac{1}{2}(\cdot,e),\] where
\(d\) stands for the orbital distance on the quaternionic hyperbolic
\(n\)-space.
\end{corollary}

\subsubsection{Organisation of the paper}
\label{sec:organisation_of_the_paper}

The paper is organized as follows. In Section
\ref{sec:optimal_compression_exponent}, we establish Theorem
\ref{da6dad}. Section
\ref{sec:a_smooth_version_of_the_harmonic_representation_of_cocycles} is
devoted to a smooth harmonic representation theorem for uniformly convex
Banach representations. Finally, in Section
\ref{sec:the_hilbert_space_case}, we prove Theorem \ref{6ea617}; the
existence part relies on the results of Section
\ref{sec:a_smooth_version_of_the_harmonic_representation_of_cocycles}.

\subsubsection{Acknowledgements}
\label{sec:acknowledgements}

We would like to thank Pierre Pansu and Alain Valette for fruitful
discussions, and John Mackay for pointing us to the work of Naor and
Peres. The first author was supported by EPSRC Standard Grant
EP/V002899/1.

\section{Optimal compression exponent}
\label{sec:optimal_compression_exponent}

Throughout this section, let \(E\) denote a \(p\)-uniformly smooth
Banach space (necessarily \(1 <p \leq 2\) as of Remark \ref{d20809}).
There exists the following dual notion.

\begin{definition}
The \emph{modulus of convexity} \(\delta_B(\varepsilon)\) of a Banach
space \(B\) is defined as:
\[\delta_B(\varepsilon) = \inf \left\{ 1 - \frac{\|u + v\|}{2} : \|u\|, \|v\| \leq 1 \text{ and } \|u - v\| \geq \varepsilon \right\}, \quad 0 < \varepsilon \leq 2.\]
\end{definition}

If \(\delta_B(\varepsilon) > 0\) for all \(\varepsilon > 0\), then \(B\)
is said to be \emph{uniformly convex}. If
\(\delta_B(\varepsilon) \geq c \varepsilon^q\) for some \(q \geq 2\),
then \(B\) is said to be \emph{\(q\)-uniformly convex}.

A Banach space \(B\) is \(p\)-uniformly smooth if and only if its dual
\(B^*\) is \(q\)-uniformly convex, where
\(\frac{1}{p} + \frac{1}{q} = 1\). This follows from the relation
\[\rho_{\mathrm{B}}(\tau)=\sup _{0 \leq \varepsilon \leq 2}\left(\tau \varepsilon / 2-\delta_{\mathrm{B}^{*}}(\varepsilon)\right) \quad(\tau>0)\]
due to Lindenstrauss \cite[Theorem~1]{Lindenstrauss1963}.

\begin{lemma}
\label{b8e205}
Let \(G\) be a group acting on \(E\) by uniformly bounded operators.
Then there exists an equivalent \(p\)-uniformly smooth, \(G\)-invariant
norm on \(E\).
\end{lemma}

\begin{proof}
Since \((E,\|\cdot\|)\) is \(p\)-uniformly smooth, \((E^*,\|\cdot\|^*)\)
is \(q\)-uniformly convex, where \(\frac{1}{p} + \frac{1}{q} = 1\).

Define an equivalent \(G\)-invariant norm on \(E^*\) by
\[\|v\|^{*'} := \sup_{g \in G} \|gv\|^*.\] As shown in the proof of
\cite[Proposition~2.3]{Bader2007}, this new norm is also \(q\)-uniformly
convex. Therefore, its dual norm on \(E^{**} = E\) defines a
\(G\)-invariant, \(p\)-uniformly smooth norm, which is equivalent to the
original norm \(\|\cdot\|\).

\end{proof}

We need a more qualitative version of the relative Vinogradov notation.

\begin{lemma}
\label{87d5ac}
Let \(G\) be a locally compact compactly generated group. Given
\(f, h: G \rightarrow \mathbb{R}_{+}\), then \[f \prec h\] if and only
if there exists a measurable, increasing map
\(\eta: \mathbb{R}_{+} \to \mathbb{R}_{+}\) with
\[\lim_{ x \to \infty } \eta(x) = \infty, \:\text{for any diverging sequence},\]
such that for all \(g \in G\), \[f(g)\eta(|g|_{S}) \leq h(g).\] Here
\(|\cdot|_{S}\) denotes the word norm w.r.t. some open relatively
compact generating set \(S\) of \(G\).
\end{lemma}

\begin{proof}
If such a map \(\eta\) exists, then for any \(k > 0\) there exists a
closed ball \(\bar{B}\) s.t. \(\eta(|g|_{S}) \geq k\) outside of this
ball. Since closed balls are compact w.r.t. \(|\cdot|_{S}\), this
implies \(f \prec h\).

Conversely, assume \(f \prec h\). By definition, for any increasing
sequence \(\{k_n\}\) of positive real numbers, there exists a sequence
of compact sets \(\{K_n\}\) such that \[
f(g) k_n \leq h(g)
\]\\
for all \(g\) outside \(K_n\). Take a sequence \(\{ B_{i} \}\) of nested
balls s.t. \(G = \bigcup_{i}B_{i}\). Let \(i_{1}\) be the minimal index
s.t. \(K_{1} \subset B_{i_{1}}\). Given \(i_{n}\), let
\(i_{n+1} = i_{n} + 1\) if \(K_{n+1} \subset B_{i_{n}}\), otherwise let
\(i_{n+1}\) be the minimal index s.t. \(K_{n+1} \subset B_{i_{n+1}}\).
Let \(r_{i}\) denote the radius of \(B_{i}\). Define a function
\(\eta: \mathbb{R}_{+} \to \mathbb{R}_{+}\) piecewise by\\
\[\eta(x) = \begin{cases}
0 & \text{for } x \leq r_{i_{1}}\\
k_{i_{n}} & \text{for } r_{i_{n}} < x \leq r_{i_{n+1}}.
\end{cases}\]

\end{proof}

\begin{proof}[Proof of Theorem \ref{da6dad}]
Thanks to Lemma \ref{b8e205} we may replace \((\pi,B)\) by an isometric
representation. Let \(S\) be any finite, symmetric generating set,
\(S \subset \Gamma\), and \(\left(W_n\right)_n\) the associated random
walk. Naor and Peres show in \cite[Theorem~2.1]{Naor2008}, that the
image of \(\left(W_n\right)_n\) under a cocycle \(b: G \to B\) satisfies
some Markov type inequality. Namely, for every time
\(n \in \mathbb{N}\),
\[\mathbb{E}\left[\left\|b\left(W_{n}\right)\right\|^{p}\right] \lesssim  n \,  \mathbb{E}\left[\left\|b\left(W_{1}\right)\right\|^{p}\right].\]

Our proof goes by contradiction. Assume \(\Gamma\) is nonamenable and
\(|g|^{\frac{1}{p}} \prec \|b(g)\|\). As \(\Gamma\) is nonamenable,
there exists \(\lambda>0\) such that \[
1=\mathbb{P}\left(\liminf _n \frac{\left|W_n\right|}{n}>\lambda\right) \leq \liminf _n \mathbb{P}\left(\left|W_n\right|>\lambda n\right)
\] (see \cite[Proposition~8.2]{Woess2000} for the existence of
\(\lambda\), the inequality follows from Fatou's Lemma.) Thus for all
measurable positive nondecreasing functions \(\eta\) : \[
\begin{aligned}
\mathbb{E}\left(\eta\left(\left|W_n\right|\right)\right) & \geq \mathbb{E}\left(\eta\left(\left|W_n\right|\right) 1_{\frac{\left|W_n\right|}{n} \geq \lambda}\right) \geq \eta(\lambda n) \mathbb{P}\left(\frac{\left|W_n\right|}{n} \geq \lambda\right) \\
& \geq \frac{1}{2} \eta(\lambda n)
\end{aligned}
\] for \(n\) large enough that only depend on \((\Gamma, S)\).

Choose \(\eta\) as in Lemma \ref{87d5ac}. It is known, that for any
random variable \(X\) and two increasing functions \(f\) and \(g\), the
covariance of \(f(X)\) and \(g(X)\) satisfies
\(\operatorname{Cov}[f(X), g(X)] \geq 0\), for details see
\cite{Schmidt2014}. It follows that \[
\begin{aligned}
\mathbb{E}\left(\left\|b\left(W_n\right)\right\|^p\right) & \geq \mathbb{E}\left(\eta^p\left(\left|W_n\right|\right)\left|W_n\right|\right) \geq_{\mathrm{Cov} \geq 0} \mathbb{E}\left(\left|W_n\right|\right) \mathbb{E}\left(\eta^p\left(\left|W_n\right|\right)\right) \\
& \geq \frac{1}{4} \lambda n \,\eta^p(\lambda n)
\end{aligned}
\] for \(n\) large enough. This contradicts the bound:
\(\mathbb{E}\left(\|b(W_{n})\|^p\right) \lesssim n\) from
\cite[Theorem~2.1]{Naor2008}.

\end{proof}


\section{A smooth version of the harmonic representation of cocycles}
\label{sec:a_smooth_version_of_the_harmonic_representation_of_cocycles}

Let \(G\) be a connected semi-simple Lie group with finite center and
\(K\) a maximal compact subgroup of \(G\). Let \(E\) be a uniformly
convex Banach space and \((\pi_{E},E)\) an isometric Banach
representation of \(G\).

A map \(F: G/K \to E\) is harmonic, if it is at least twice continuously
differentiable and \(\Delta F = 0\) for the Laplacian \(\Delta\) on
\(G/K\).

The goal of this section is to show that any cocycle in a uniformly
convex Banach representation with a spectral gap has a harmonic
representative.

\begin{definition}
Let \(G\) be a locally compact topological group, and \((\pi, B)\) an
isometric Banach representation.

\begin{enumerate}
\def\labelenumi{\arabic{enumi}.}
\tightlist
\item
  For a subset \(Q\) of \(G\) and real number \(\varepsilon>0\), a
  vector \(v\) in \(B\) is \emph{\((Q, \varepsilon)\)-invariant} if
  \[\sup _{g \in Q}\|\pi(g) v-v\|<\varepsilon\|v\|.\]
\item
  The representation \((\pi, B)\) \emph{almost has invariant vectors} if
  it has \((Q, \varepsilon)\)-invariant vectors for every compact subset
  \(Q\) of \(G\) and every \(\varepsilon>0\).
\item
  The representation \((\pi, B)\) \emph{has non-zero invariant vectors}
  if there exists \(v \neq 0\) in \(B\) such that \(\pi(g) v=v\) for all
  \(g \in G\). If this holds, we write \(1_G \subset \pi\).
\end{enumerate}
\end{definition}

The representation \((\pi,B)\) \emph{has no almost invariant vectors} if
there exists a compact subset \(Q\) in \(G\) and a constant \(\kappa>0\)
such that for every \(v \in B\),
\[\sup _{g \in Q}\left\|\pi(g)v- v\right\| \geq \kappa \,\|v\| .\] The
pair \((Q, \kappa)\) is called a \emph{Kazhdan pair} for \(\pi\).

\begin{definition}
An isometric representation \((\pi,B)\) on a Banach space \(B\) is
\emph{complemented} if \[B=B^\pi \oplus B_\pi,\] where \(B^{\pi}\) is
the space of fixed vectors.
\end{definition}

\begin{remark}
\label{1e3f00}
Any isometric representation \(\pi\) on super-reflexive Banach spaces are complemented \cite[Proposition~2.6]{Bader2007}.
\end{remark}

\begin{definition}
\label{3ad315}
We say that a complemented representation \((\pi,B)\) \emph{has a
spectral gap}, if \((\pi,B_{\pi})\) has no almost invariant vectors.
\end{definition}

Let us assume from now on that \((\pi_{E},E)\) has a spectral gap.

\begin{remark}[{\cite[Theorem~1.1]{Drutu2019}}]
\label{343c7b}
Let \(G\) be a locally compact group and \((\pi_E,E)\) a complemented
Banach representation of \(G\). Then \((\pi_E,E)\) has a spectral gap if
and only if there exists a compactly supported probability measure
\(\mu\) on \(G\) s.t. the Markov operator
\[\pi_E(\mu) = \int_{G} \pi_E(g) \, d\mu\] restricted to \(B_{\pi}\) has a
spectral gap, i.e. \[\|\pi_E(\mu)|_{E_{\pi}}\|_{\mathrm{op}} < 1.\]
\end{remark}

In our semi-simple setting, we improve \(\mu\) slightly.

\begin{proposition}
\label{5c6adf}
There exists a probability density
\(\rho \in C^{0}_{c}(K\backslash G/K)\) with support generating a dense
subgroup of \(G\) such that \[\|\pi_{E}(\rho)\|_{\mathrm{op}}<1.\]
\end{proposition}

\begin{proof}
Using Remark \ref{1e3f00}, \(\pi_{E}\) decomposes as
\(E_{\pi_{E}}\oplus E^{\pi_{E}}\) and the cocycle \(b\) splits
accordingly. As \(G\) is semi-simple, the second summand must be
trivial. Indeed otherwise \(b|_{E^{\pi_{E}}}\) would induce a morphism
between \(G\) and \((E^{\pi_{E}},+)\), which contradicts the fact that
\(\mathrm{Hom}(G,\mathbb{R})=0\). One can therefore assume that
\(E=E_{\pi_{E}}\).

Let \(Q\) be a compact generating set of \(G\). Pick a pair
\((\alpha_Q,\beta_Q)\) of functions in \(C_c^0(K\backslash G/K)\) such
that \(\alpha_Q,\beta_Q\ge0\), \(\alpha_Q(e)>0\), \(\int_G\alpha_Q=1\),
and \(g \cdot \alpha_Q(a)\le \beta_Q(a)\) for all \(g\in Q\) and
\(a\in G\). Notice that \(\beta_Q\) is supported on \(Q\).

Given a pair \((\alpha_Q,\beta_Q)\) we denote by
\(\rho_Q\in C_c^0(K\backslash G/K)\) the continuous probability density
on \(G\) defined by
\[\rho_Q=\frac{\alpha_Q+\beta_Q}{\int_G(\alpha_Q+\beta_Q)}.\]

By assumption, there exists \(\varepsilon>0\) that makes
\((Q,\varepsilon)\) a Kazhdan pair for \((\pi_E,E)\). We remark that \(\pi_{E}\) has
almost invariant vectors if and only if it has a
\((Q, \varepsilon)\)-invariant vector for every \(\varepsilon>0\). It
follows from arguments similar to \cite[Theorem~3.4]{Drutu2019}, that
\(\|\pi_{E}(\rho)\|_{\mathrm{op}}<1\).

\end{proof}

\begin{definition}
Let \(G\) be a topological group, \(B\) a Banach space. An \emph{affine
representation} of \(G\) on \(B\) is a continuous group homomorphism
\[\alpha: G \rightarrow \operatorname{Aff}(B) = B \rtimes \mathrm{GL}(B),\]
where \(\operatorname{Aff}(B)\) is equipped with the product topology
coming from the topology of \(B\) and the compact open topology on
\(\mathrm{GL}(B)\).
\end{definition}

A continuous map \(F: G/K \to E\) is \(G\)-equivariant if there exists
an affine representation \(\alpha\) with linear part \(\pi_{E}\), such
that
\[F = \alpha(g) \circ F \circ \lambda(g^{-1}) \quad \forall g \in G,\]
where \(\lambda(g)\) denote the automorphism of \(G/K\) of
left-translation by \(g \in G\).

\begin{proposition}
\label{74e9c2}
Given \(b\in Z^1(G,\pi_{E})\), there exists \(b'\in B^1(G,\pi_{E})\)
such that \(b_{K} = b+b'\) is smooth and induces a \(G\)-equivariant map
\(F_{b_{K}}:G/K\rightarrow E\) such that \(\Delta F_{b_{K}}=0\).
\end{proposition}

\begin{proof}
Using Proposition \ref{5c6adf}, there exists a probability density
\(\rho \in C^{0}_{c}(K\backslash G/K)\) such that
\(\|\pi_{E}(\rho)\|_{\mathrm{op}}<1\). Therefore the affine
transformation \(\alpha(\rho)=\pi_{E}(\rho)+b(\rho)\) has a unique fixed
point \(x_1\in E\). Moreover, as \((G,K)\) form a Gelfand pair (see for
example \cite[Proposition~6.1.3]{Dijk2009}), the space of probability
densities \(C^{0}_{c}(K\backslash G/K)\) is commutative for convolution
and, by uniqueness of the fixed point of \(\alpha(\rho)\), the point
\(x_1\) is fixed by every probability density in
\(C^{0}_{c}(K\backslash G/K)\).

Define the coboundary \(b'(g) = \pi_{E}(g)x_{1} - x_{1}\) and
\(b_{K} = b + b'\). Then for any probability density
\(\rho \in C^{0}_{c}(K\backslash G/K)\), \(b_{K}(\rho) = 0\). By the
\(K\)-invariance of \(\rho\), for any \(k \in K\), \[
\begin{split}
b_{K}(\rho) &= \int_{G} b_{K}(kg)\rho(g)\,dg \\
&= b_{K}(k) + \pi_{E}(k) b_{K}(\rho).
\end{split}
\] Thus \(b_{K}|_{K} = 0\) and \(b_{K}\) induces a map on \(G/K\)
denoted \(F_{b_{K}}\).

More generally, for any probability density
\(\rho \in C^{0}_{c}(K\backslash G/K)\), \begin{equation}
\label{5eaaeb}
\begin{split}b_{K}(g) &= \int_{G} b_{K}(gh)\rho(h)\,dh - \pi_{E}(g) b_{K}(\rho)\\&=\int_G b_{K}(h) \rho(g^{-1}h)\,dh.\end{split}\end{equation}

We may choose \(\rho\) to be smooth and thus \(F_{b_{K}}\) is smooth on
\(G/K\).

The mean value property (\ref{5eaaeb}) implies that \(F\) is annihilated
by \(G\)-invariant operators acting on appropriate function spaces
\(\mathcal{F}\) as defined in \cite[pp.~156]{Bekka2008}. \(F_{b_{K}}\)
belongs to the \(G\)-invariant Fréchet space \(C^{\infty}(G/K,B)\).
Recall that \(f \to 0\) in \(C^{\infty}(G/K,B)\) if and only if \(f\)
and all its derivatives tend to \(0\) uniformly on compact subsets on
\(G/K\). Point evaluations
\(\operatorname{ev}_{x}: C^{\infty}(G/K,B) \to B\) and orbit maps by
left-translation on \(C^{\infty}(G/K,B)\) are continuous. Thus
\(C^{\infty}(G/K,B)\) satisfies the conditions on \(\mathcal{F}\) in
\cite[pp.~156]{Bekka2008}. By \cite[Proposition~3.3.8]{Bekka2008}
\(D F_b=0\) for every \(G\)-invariant differential operator \(D\) acting
on \(C^{\infty}(G/K,B)\).

Clearly, the map \(F_{b_{K}}\) is \(G\)-equivariant w.r.t. the isometric
affine Banach action \(\alpha_{K} = \pi_{E} + b_{K}\).

\end{proof}

\section{The Hilbert space case}
\label{sec:the_hilbert_space_case}

Let \(G\) be a connected semisimple Lie group with finite center, and
let \(K \leq G\) be a maximal compact subgroup. Let \((\pi,H)\) be a
uniformly bounded Hilbert representation of \(G\).

In this section, we exploit the harmonicity a map \(F: G/K \to H\) to
relate its growth to the infinitesimal behavior of \(F\). This
relationship is made explicit in Subsection
\ref{sec:real_rank_1_proof_of_theorem_ref6ea617}, where it is analyzed
using ordinary differential equation techniques.

In the Hilbert space setting, the infinitesimal behavior of \(F\) is
captured by the Hilbert--Schmidt norm of its differential at points
\(x \in G/K\). Our first objective is therefore to obtain global control
on these Hilbert--Schmidt norms, which will serve as the main analytic
input for the subsequent growth estimates.

\begin{lemma}
\label{428ac3}
Let \(F: G/K \to H\) be a harmonic map, then for every \(x \in G/K\),
\[\Delta \|F(x)\|^2 = -2\left\|d_{x} F\right\|_{\mathrm{HS}}^2,\] where
\(\left\|d_{x} F\right\|_{\mathrm{HS}}=\left(\operatorname{Tr}\left(\left(d_{x} F\right)^* d_{x}F\right)\right)^{1 / 2}\)
is the Hilbert-Schmidt norm of the differential of \(F\) at \(x\).
\end{lemma}

\begin{proof}
This follows directly from the equation
\[\Delta \|F(x)\|^2 = 2 \langle\Delta F(x), F(x)\rangle -2\left\|d_{x} F\right\|_{\mathrm{HS}}^2.\]
For details see \cite[Proposition~3.3.14.]{Bekka2008}.

\end{proof}

We may assume that \(K\) acts unitarily on \(H\) by averaging the inner
product via the normalized Haar measure on \(K\).

\begin{proposition}
\label{2a47dc}
Let \(F: G/K \to H\) be a nonconstant \(G\)-equivariant harmonic map.
Then the Hilbert-Schmidt norm of the operator
\(d_{x}F: T_{x}(G/K) \to H\) is uniformly bounded, i.e.~there exists
\(c > 0\) s.t. for all \(x \in G/K\),
\[ \frac{\sqrt{ \operatorname{dim}(G/K) }}{c} \leq \|d_{x}F\|_{HS} \leq c \, \sqrt{ \operatorname{dim}(G/K) }.\]
\end{proposition}

Before proceeding with the proof, we briefly recall the connection
between the differentiable structure of \(G/K\) and the Lie algebra of
\(G\), for details we refer to \cite[Chapter~III,~7]{Helgason2001} and
\cite[Chapter~V,~6]{Helgason2001}.

Let \(\mathfrak{g}\) represent the Lie algebra of \(G\), and let
\(\mathfrak{k}\) be the subalgebra associated with \(K\). There exists a
Cartan decomposition \(\mathfrak{k} \oplus \mathfrak{p}\) of
\(\mathfrak{g}\). For every element \(X \in \mathfrak{p}\), there exists
a corresponding tangent vector \(D_X \in T_{x_0}(G / K)\), where
\(x_0 = K\), given by
\(D_X f(x_0) = \frac{d}{dt} f(\exp(tX) x_0)|_{t=0}\),
\(\forall f \in C^{\infty}(G / K)\). The mapping \(X \mapsto D_X\)
defines a linear isomorphism between \(\mathfrak{p}\) and
\(T_{x_0}(G / K)\), enabling an identification of \(T_{x_0}(G / K)\)
with \(\mathfrak{p}\).

For any \(k \in K\), the linear automorphism \(d\lambda(k)_{x_0}\) of
\(T_{x_0}(G / K)\) corresponds to the inner automorphism
\(\operatorname{Ad}(k): \mathfrak{p} \to \mathfrak{p}\). This
relationship is established as follows: \[
\begin{aligned}
\left.\frac{d}{dt} f\left(k \exp(tX) x_0\right)\right|_{t=0} 
& = \left.\frac{d}{dt} f\left(k \exp(tX) k^{-1} x_0\right)\right|_{t=0} \\
& = \left.\frac{d}{dt} f\left(\exp\left(t \operatorname{Ad}(k)X\right) x_0\right)\right|_{t=0},
\end{aligned}
\] for all \(f \in C^{\infty}(G / K)\).

\begin{proof}[Proof of Proposition \ref{2a47dc}]
The tangent space \(T_{x_0}(G / K)\) is identified with \(\mathfrak{p}\)
as previously described, and the tangent space at any vector \(v \in H\)
is similarly identified with \(H\). Using the fact that \(F\) is
\(G\)-equivariant, we deduce the following for all \(k \in K\) and
\(X \in \mathfrak{p}\): \[
\begin{aligned}
d_{x_{0}} F(\operatorname{Ad}(k) X) & =\left.\frac{d}{d t} F\left(\exp(t \operatorname{Ad}(k) X) x_0\right)\right|_{t=0} \\
& =\left.\frac{d}{d t} F\left(k \exp(t X) x_0\right)\right|_{t=0} \\
& =\left.\frac{d}{d t}\left(\pi(k) F\left(\exp(t X) x_0\right)\right)\right|_{t=0} \\
& =\pi(k) \left.\frac{d}{d t} F\left(\exp(t X) x_0\right)\right|_{t=0} \\
& =\pi(k) \, d_{x_{0}} F(X).
\end{aligned}
\] Since \(K\) acts unitarily on \(H\),
\(\|\pi(k)d_{x_{0}} F(X)\| = \|d_{x_{0}} F(X)\|\) for all \(k \in K\).
Thus, the symmetric bilinear form \(B_{x_{0}}\) on
\(T_{x_0}(G / K) \cong \mathfrak{p}\), defined by \[
B_{x_{0}}(X, Y) = \left\langle d_{x_{0}} F(X), d_{x_{0}} F(Y) \right\rangle,
\] is invariant under the action of \(\operatorname{Ad}(K)\).

We claim that \(B_{x_{0}}\) is positive definite. Assume that there
exists a non-zero \(Z \in T_{x_{0}}(G/K)\) s.t.
\(\|d_{x_{0}}F(Z)\| = 0\). Then by the \(\operatorname{Ad}(K)\)
invariance of \(B_{x_{0}}\), \(d_{x_{0}}F\) vanishes on the whole
tangent space \(T_{x_{0}}(G/K)\). By \(G\)-equivariance of \(F\), for
any \(g \in G\) and \(x = g x_0\), it holds that \begin{equation}
\label{9e84cd}
d_{x} F = \pi(g) \,d_{x_{0}} F\,d_{x} \lambda\left(g^{-1}\right).\end{equation}

Since \((\pi,H)\) is uniformly bounded,
\[\left\|d_{x} F\left(d_{x_{0}} \lambda(g) (X)\right)\right\| \lesssim \left\|d_{x_{0}} F(X)\right\| = 0\]
for all \(X \in \mathfrak{p}\). Thus \(F\) would be constant, a
contradiction.

By norm equivalence, there exists \(c > 0\), s.t.
\[\frac{\|X\|_{2}^2}{c^{2}}\leq B_{x_{0}}(X,X) \leq c^{2} \,\|X\|_{2}^2 \quad \text{for all } X \in T_{x_0}(G / K).\]

We claim that \(c\) can be chosen independently of the base-point
\(x_{0}\). Indeed, by \(G\)-equivariance of \(F\), for any \(g \in G\)
and \(x = g x_0\), we have \[
\begin{aligned}
B_{x}(d_{x_{0}} \lambda(g) (X),d_{x_{0}} \lambda(g) (X) )
&=\left\|d_{x} F\left(d_{x_{0}} \lambda(g) (X)\right)\right\|^{2} \\
& \leq C \, \left\|d_{x_{0}} F(X)\right\|^{2} \\
& = C \, B_{x_{0}}(X,X) \\
& \leq C\,c^{2} \, \left\|X\right\|_2^{2}
\end{aligned}
\] for all \(X \in \mathfrak{p}\). A lower bound follows analogously.
The uniform bounds on the Hilbert-Schmidt norm follows by summing over
an orthonormal basis.

\end{proof}

\subsection{Real rank 1 - Proof of Theorem \ref{6ea617}}
\label{sec:real_rank_1_proof_of_theorem_ref6ea617}

If we further require \(G\) to be of real rank one, then the coset space
\(G/K\), can be identifies with either the real \(n\)-hyperbolic space,
the complex \(n\)-hyperbolic space, the quaternionic \(n\)-hyperbolic
space, or the hyperbolic plane over the octonionic numbers. For details
see for example \cite[Theorem~6.105]{Knapp2002} and
\cite[Chapter~3]{Cherix2001}. We denote by \(\mathbb{K}\) the real,
complex, quaternionic or octonionic numbers.

The growth of a \(G\)-equivariant map \(F: G/K \to H\) is radial in the
sense that there exists a function
\(\varphi: \mathbb{R}_{+} \to \mathbb{R}_{+}\), s.t. for every
\(x \in G/K\), and \(r = d(x,x_{0})\) \[\|F(x)\|^{2} = \varphi(r).\]
From equation \ref{9e84cd} it follows that there also exists a function
\(\eta: \mathbb{R}_{+} \to \mathbb{R}_{+}\), s.t.
\[\|d_{x}F(x)\|_{HS} = \eta(r).\]

Thus for a \(G\)-equivariant harmonic map \(F: G/K \to H\),
\[\Delta \|F(x)\|^{2}=-\frac{d^2 \varphi}{d r^2}-m(r) \frac{d \varphi}{d r} = -2\eta(r)\]
where \(m(r)=m_1 \operatorname{coth} r+2 m_2 \operatorname{coth} 2 r\)
and
\(m_1=k(n-1), \, m_2=k-1, \, k=\operatorname{dim}_{\mathbb{R}} \mathbb{K}\)
(see \cite[Chapter~II,~Section~3]{Helgason1984} and Lemma \ref{428ac3}).

\begin{proof}[Proof of Theorem \ref{6ea617}]
Since \(H^1(\pi,H)\neq0\), there exists an unbounded cocycle \(b\). By
Proposition \ref{74e9c2} there exists \(b' \in B^{1}(G,\pi)\) such that
\(b_{K} = b + b'\) induces a harmonic map \(F: G/K \to H\) that is
\(G\)-equivariant w.r.t. \(\alpha = \pi + b_{K}\). Since \(b\) is
unbounded, \(F\) is nonconstant.

Choosing \(\varphi\) and \(\eta\) as above leads to the equation:
\[ \varphi''(r) + m(r) \, \varphi'(r) = 2 \eta(r). \] By setting
\(\psi = \varphi'\), we derive a first-order ordinary differential
equation for \(\psi\), which is solved using the method of variation of
constants. The general solution to the associated homogeneous equation
is a constant multiple of the function:
\[ \psi_0(r) = (\sinh r)^{-m_1} (\sinh 2r)^{-m_2}. \] Thus, a particular
solution to the nonhomogeneous equation is:
\[ \psi(r) = 2\psi_0(r) \int_0^r \frac{\eta(s)}{\psi_0(s)} \, ds. \]
Now, setting:
\[ f(r) = \frac{1}{\psi_0(r)} = (\sinh r)^{m_1} (\sinh 2r)^{m_2}, \] we
obtain for the limit superior (analogously for the limit inferior)
\[ \begin{split} \limsup_{r \to \infty} \frac{1}{f(r)} \int_0^r \eta(s) f(s) \, ds &= \limsup_{s \to \infty} \eta(s) \lim_{ r \to \infty } \frac{f(r)}{f'(r)} \\& = \limsup_{s \to \infty} \frac{\eta(s)}{m_1 + 2 m_2}. \end{split} \]
Thus, we conclude together with Proposition \ref{2a47dc} that there
exists a \(c > 0\) s.t.
\[ \frac{2n }{m_1 + 2 m_2} \frac{1}{c}  + o(1) \leq \psi(r) \leq \frac{2n }{m_1 + 2 m_2} \, c + o(1) \]
as \(r \to \infty\). Integrating implies \[
\|F(\cdot)\|^{2} \sim d(\cdot,x_{0}).
\]

\end{proof}

\subsubsection{Proof of Corollary \ref{99b6bf} }
\label{sec:proof_of_corollary_ref99b6bf}

After identifying \(\operatorname{Sp}(n,1)/K\) with quaternionic
hyperbolic \(n\)-space, we denote by \(\partial X\) the ideal boundary
of \(\operatorname{Sp}(n,1)/K\). For background and details, see
\cite[Chapter~3]{Cherix2001}.

Let \(\pi\) be the natural representation of \(\operatorname{Sp}(n,1)\)
on \(C^{\infty}(\partial X)\) given by left translation. The subspace
\(\mathbb{C}1_{\partial X}\) of constant functions is a trivial
subrepresentation of \(\pi\). We denote by \(\pi_0\) the induced
quotient representation on
\[C^{\infty}(\partial X) \big/ \mathbb{C}1_{\partial X}.\] Nishikawa
showed that there exists a Euclidean norm on this quotient space with
respect to which \(\pi_0\) is continuous and uniformly bounded; see
\cite[Proposition~2.13]{Nishikawa2020_preprint}. Moreover, he proves
that \(\pi_0\) admits a nontrivial cocycle.

In order to obtain a harmonic representative of the cocycle (Proposition
\ref{74e9c2}), we show that \(\pi_{0}\) has a spectral gap. This follows
from the lemma below together with the fact that the only functions in
\(C^{\infty}(\partial X)\) invariant under left translation are the
constant functions.

Recall that by Lemma \ref{b8e205}, any uniformly bounded Euclidean norm
\(\|\cdot\|\) can be modified into a \(G\)-invariant
\(\frac{1}{2}\)-uniformly smooth Banach norm.

\begin{lemma}
\label{418062}
Let \((\pi,B)\) be an isometric Banach space representation of a locally
compact compactly generated group \(G\). Let \(K\) be a compact subgroup
of \(G\). If the space of vectors \(B^K\) that are fixed under the
action of \(K\), has finite dimension, then either \((\pi,B)\) does not
have almost invariant vectors or \(1\subset \pi\).
\end{lemma}

\begin{proof}
Assume \(G\) almost has invariant vectors. Let \(Q\) be a compact
generating set of \(G\). Let \[P=\int_K \pi(k)\,dk\] and observe that
since \([\pi(k),P]=0\) for all \(k\in K\), \(P\) is the projection onto
the \(K\)-invariant vectors in \(B\). If \((b_n)_n\) is a normalized
sequence of \(\left( Q, \frac{1}{n} \right)\)-invariant vectors, then
\[\|Pb_n-b_n\|\le\int_K\|\pi(k)b_n-b_n\|\,dk = o(1),\] and thus
\(\|Pb_{n}\|\) is bounded away from \(0\) for \(n\) large enough. Let
\(b_n'=\frac{1}{\|Pb_n\|}Pb_n\in B^K\) for \(n\) large enough and
observe that
\begin{align*}
\|\pi(g)b_n'-b_n'\|
&=\frac{1}{\|Pb_n\|}\|\pi(g)Pb_n-Pb_n\| \\
&\le \frac{2}{\|Pb_n\|}\|Pb_n-b_n\|+\frac{1}{\|Pb_n\|}\|\pi(g)b_n-b_n\|
= o(1)
\end{align*}
uniformly on compact sets of \(G\). On the other hand, \(B^K\) is
finite dimensional, and one can therefore assume that \((b_n')_n\)
converges to \(b\in B^K\) with \(b\neq0\). It follows that \(b\) is
invariant under \(G\).

\end{proof}

\bibliographystyle{plain}
\bibliography{extracted.bib}
\end{document}